\documentclass[letterpaper,12pt,reqno]{amsart}
\usepackage{graphicx,amsmath,amssymb, amsthm}

\usepackage{srcltx}
\usepackage[latin1]{inputenc}
\usepackage[T1]{fontenc}

\usepackage{fullpage,hyperref}
\usepackage{url}

\newtheorem{X}{X}[section]
\newtheorem{corollary}[X]{Corollary}

\newtheorem{lemma}[X]{Lemma}
\newtheorem{proposition}[X]{Proposition}
\newtheorem{theorem}[X]{Theorem}

\theoremstyle{definition}
\newtheorem*{remark*}{Remark}

\renewcommand{\a}{\overrightarrow{a}}

\newcommand{\be}{\overrightarrow{\beta}}

\newcommand{\Irr}{\text{Irr}}

\newcommand{\Gal}{{\rm Gal}}
\newcommand{\Q}{\mathbb Q}
\renewcommand{\d}{{\rm d}}

\renewcommand{\be}{\begin{equation}}
\newcommand{\ee}{\end{equation}}
\newcommand\bea{\begin{eqnarray}}
\newcommand\eea{\end{eqnarray}}
\newcommand\bi{\begin{itemize}}
\newcommand\ei{\end{itemize}}
\newcommand\ben{\begin{enumerate}}
\newcommand\een{\end{enumerate}}
\newcommand\bc{\begin{center}}
\newcommand\ec{\end{center}}
\newcommand\ba{\begin{array}}
\newcommand\ea{\end{array}}

\newcommand{\Z}{\mathbb Z}

\title{Unconditional Chebyshev biases in number fields}
\author{Daniel Fiorilli and Florent Jouve}
\date{\today}
\address{Univ. Paris-Saclay, CNRS, Laboratoire de mathématiques d'Orsay, 91405, Orsay, France.} 
\email{daniel.fiorilli@universite-paris-saclay.fr}
\address{Univ. Bordeaux, CNRS, Bordeaux INP, IMB, UMR 5251,  F-33400, Talence, France.}
\email{florent.jouve@math.u-bordeaux.fr}

\begin{document}

\begin{abstract}

Prime counting functions are believed to exhibit, in various contexts, discrepancies beyond what famous equidistribution results predict; this phenomenon is known as Chebyshev's bias.
%, since he was the first to discover empirically in 1853 the preponderance of the residue class $3$ in the distribution modulo 4 of prime numbers. 
Rubinstein and Sarnak have developed a framework which allows to conditionally quantify biases in the distribution of primes in general arithmetic progressions.  Their analysis has been generalized by Ng to the context of the Chebotarev density theorem,
%distribution of Frobenius elements at prime ideals in Galois extensions of number fields,
 under the assumption of the Artin holomorphy conjecture, the Generalized Riemann Hypothesis, as well as a linear independence hypothesis on the zeros of Artin $L$-functions. 
%In this paper we show the existence of an infinite family of Galois extensions $L/K$ and associated conjugacy classes $C_1,C_2\subset \Gal(L/K)$ for which the densities can be computed unconditionally.
 In this paper we show unconditionally the occurence of  extreme biases in this context. These biases lie far beyond what the strongest effective forms of the Chebotarev density theorem can predict. More precisely, we prove the existence of an infinite family of Galois extensions and associated conjugacy classes $C_1,C_2\subset \Gal(L/K)$ of same size such that the number of prime ideals of norm up to $x$ with Frobenius conjugacy class $C_1$ always exceeds that of Frobenius conjugacy class $C_2$, for every large enough $x$. 
 %A crucial input comes from 
 A key argument in our proof relies on features of certain subgroups of symmetric groups which enable us to circumvent the need for unproven properties of  zeros of Artin $L$-functions.
\end{abstract}
\maketitle

\section{Introduction and statement of results}
\label{section intro}

%Since Chebyshev's letter to Fuss~\cite{Ch}, many variants and generalization of his observation have been explored. 

In 1853, Chebyshev noticed in a letter to Fuss~\cite{Ch} that there seems to exist a bias in the distribution of primes modulo $4$, that is in most intervals of the form $[2,x]$, there appears to be more primes of the form $4n+3$ than of the form $4n+1$.
 It turns out that the specific statements made in Chebyshev's letter are quite deep: the second is equivalent to the Riemann hypothesis for $L(s,\chi_{-4})$, and the first can be made explicit under an additional linear independence hypothesis on the zeros of $L(s,\chi_{-4})$.
Chebyshev's observation has been widely generalized over the years; notably, Rubinstein and Sarnak~\cite{RS}  have shown that for two invertible residue classes $a$ and $b$ modulo $q$, there exists a bias towards $a$ (that is $\pi(x;q,a)>\pi(x;q,b)$ is true more often than $\pi(x;q,a)<\pi(x;q,b)$) if and only if $b$ is a quadratic residue and $a$ is a non-quadratic residue. These theoretical results, as well as the numerical determinations of the bias in the paper, are conditional on the generalized Riemann hypothesis and a linear independence hypothesis on the non-trivial zeros of Dirichlet $L$-functions. In the same paper~\cite[\S 5]{RS}, the authors mention several possible generalizations including biases in the distribution of prime ideals in Galois extensions of number fields. This context was explored by Ng in his Ph.D. thesis~\cite{Ng}. 
%In his 1853 letter to Fuss~(\cite{Ch}), Chebyshev observed a bias in the distribution of primes in residue classes.
%%makes the observation, based on numerical computations, that for most integers $x$ one has $
%%\pi(x;4,3)>\pi(x;4,1)
%%$ where $\pi(x;q,a)$ denotes the number of primes up to $x$ that are $a$ (mod $q$).
% The general phenomenon is that if $a$ is a nonsquare (mod $q$) and $b$ is a square (mod $q$), then there tend to be more primes congruent to $a$ (mod $q$) than $b$ (mod $q$) in initial intervals of the positive integers; more succinctly, there is a tendency for $\pi(x; q,a)$, that is the number of primes less than $x$ that are $a$ (mod $q$), to exceed $\pi(x; q, b)$. Rubinstein and Sarnak~(\cite{RS}) defined $\delta(q; a, b)$ to be the logarithmic density of the set of positive real numbers $x$ for which this inequality holds; intuitively, $\delta(q; a, b)$ is the ``probability'' that $\pi(x; q, a) > \pi(x; q, b)$ when $x$ is ``chosen randomly''. 
%In the present paper we c
Consider a Galois extension $L/K$ of number fields, a congugacy class $C\subset G= \Gal(L/K)$, and define the Frobenius counting function 
$$ \pi(x;L/K,C):=\sum_{\substack{\mathfrak p\triangleleft \mathcal O_K \text{ unram.}\\ \mathcal N\mathfrak p \leq x \\ {\rm Frob}_\mathfrak p =C}} 1,  $$
where ${\rm Frob}_\mathfrak p$ denotes the Frobenius conjugacy class associated to the unramified prime ideal $\mathfrak p$, and $\mathcal N\mathfrak p= | \mathcal O_K/\mathfrak p|$ denotes its norm.
The Chebotarev density theorem asserts that $\pi(x;L/K,C)\sim \frac{|C|}{|G|} \int_2^x \frac{\d t}{\log t}$.
%As a generalization of their work on the discrepancies of primes in  arithmetic progressions, Rubinstein and Sarnak~\cite[\S5]{RS} suggested to study the distribution of primes in Frobenius sets.
 More precisely, one is interested in understanding the size of the sets 
$$P_{L/K;C_1,C_2}:=\{ x \in \mathbb R_{\geq 1}\colon  |C_2|\pi(x;L/K,C_1) > |C_1|\pi(x;L/K,C_2)\}. $$
%in $\mathbb R_{\geq 1} $. 
Ng~\cite{Ng} has shown under Artin's holomorphy conjecture, GRH, as well as a linear independence hypothesis on the set of zeros of Artin $L$-functions, that the set $P_{L/K;C_1,C_2}$ admits a logarithmic density, that is the limit
$$ \delta(P_{L/K;C_1,C_2}):= \lim_{X\rightarrow \infty} \frac 1{\log X} \int_{\substack{ 1\leq x \leq X\\ x\in P_{L/K;C_1,C_2}}} \frac{\d x }{x}$$
exists. Moreover, he computed this density in several explicit extensions, under the same hypotheses. 

The goal of this paper is to show \emph{unconditionally} the existence of the density $\delta(P_{L/K;C_1,C_2})$ in some families of
% we show that there exists
  extensions and for specific conjugacy classes.
  %for which it is possible to show . 
  More precisely, we will exhibit a sufficient group-theoretic criterion on $G=\Gal(L/K)$ which implies in particular that $\delta(P_{L/K;C_1,C_2})=1$.
  % which will guarantee this existence. 
  This will involve the class function $r_G:G\rightarrow \mathbb C$ defined by 
$$ r_G(g) :=  \# \{ h\in G \colon h^2=g\}.$$
We will require $L/\Q$ to be Galois, and for a conjugacy class $C\subset G$ we will denote by $C^+$ the unique conjugacy class of $G^+:=\Gal(L/\Q)$ which contains $C$. Explicitly, 
% and we will extend conjugacy classes $C\subset G$ to well-defined
%\footnote{See Remark~\ref{remark C+ is a cc}. Note that we may not apply this definition to $G$ itself, since it is not a conjugacy class (unless $|G|=1$). } 
%conjugacy classes $C^+\subset G^+:=\Gal(L/\Q)$ by setting
\begin{equation}
C^+:=\bigcup_{a\in G^+} a C a^{-1}.
\label{equation definition C+}
\end{equation} 

\begin{theorem}
\label{theorem main}
Let $L/K$ be an extension of number fields for which $L/\Q$ is Galois. Assume that the conjugacy classes $C_1,C_2\subset G=\Gal(L/K)$ are such that $C_1^+=C_2^+$, but $ r_G(g_{C_1}) < r_G(g_{C_2}) $, where $g_{C_i}$ is a representative of $C_i$. Then, for all large enough $x$ we have the inequality $ |C_2|\pi(x,L/K,C_1) >|C_1| \pi(x,L/K,C_2) $. In particular, the set $P_{L/K;C_1,C_2}$ has natural (and logarithmic) density equal to $1$.
\end{theorem}

\begin{remark*}
The fact that the natural density of $P_{L/K;C_1,C_2}$ exists in Theorem~\ref{theorem main} is remarkable since it is widely believed that in the classical case of primes in arithmetic progressions as well as in the more general case of Galois extensions of number fields, the logarithmic density is the appropriate notion to work with. In general one cannot expect natural densities to exist 
% when investigating Chebyshev type phenomena 
(see~\cite{Ka}, as well as~\cite[p. 174]{RS} and the references therein).

Note also that in Theorem~\ref{theorem main}, one can further impose $C_1$ and $C_2$ to have the same size. Indeed, we will see in the proof of Proposition~\ref{proposition main} (see~\S\ref{sec:groups}) that there exists families of examples in which the group $G$ is abelian.
%the group $G$ that we construct is abelian.
\end{remark*}

Next we state a group theoretic result showing that the hypotheses of Theorem~\ref{theorem main} are satified by infinitely many couples $(G,G^+)$ and associated conjugacy classes $C_1,C_2 \subset G$.
\begin{proposition}
\label{proposition main}
For $n\geq 8$ the symmetric group $G^+=S_n$ admits a subgroup $G$ which contains conjugacy classes $C_1$, $C_2$ satisfying $C_1^+=C_2^+$, but $ r_G(g_{C_1}) < r_G(g_{C_2}) $, where $g_{C_i}\in C_i$ ($i=1,2$).
%Let $L/K$ be an extension of number fields for which $L/\Q$ is Galois. Assume that the conjugacy classes $C_1,C_2\subset G=\Gal(L/K)$ are such that $C_1^+=C_2^+$, but $ r_G(g_{C_1}) < r_G(g_{C_2}) $, where $g_{C_i}$ is a representative of $C_i$. Then, for all large enough $x$ we have the inequality $ |C_2|\pi(x,L/K,C_1) >|C_1| \pi(x,L/K,C_2) $. In particular, the set $P_{L/K;C_1,C_2}$ has natural (and logarithmic) density equal to $1$.
\end{proposition}

The combination of Theorem~\ref{theorem main}, Proposition~\ref{proposition main} and the fact going back to Hilbert that the inverse Galois problem over $\Q$ is solved for the symmetric group $S_n$ immediately yields the following consequence. 

\begin{corollary}\label{cor}
There exists infinitely many Galois extensions $L/K$ and conjugacy classes $C_1,C_2 \subset \Gal(L/K)$ for which $\delta(P_{L/K;C_1,C_2})=1$.
\end{corollary}

The paper is organized as follows. Section~\ref{sec:groups} is devoted to the group theoretic aspects of our main result. In particular we prove Proposition~\ref{proposition main} and discuss generalizations and related questions. In Section~\ref{sec:mainproof}, we prove Theorem~\ref{theorem main}. We conclude the paper with Section~\ref{sec:numerics} which is devoted to numerical computations and illustrations of Theorem~\ref{theorem main}.

\section*{Acknowledgments}

 Experiments presented in this paper were carried out using the PlaFRIM experimental testbed, supported by Inria, CNRS (LABRI and IMB), Université de Bordeaux, Bordeaux INP and Conseil Régional d'Aquitaine (see \url{https://www.plafrim.fr/}). We thank Bill Allombert for his insights and for providing us with the {\tt pari/gp} code and the data needed for this project. We also thank Mounir Hayani for very inspiring remarks. Finally we thank the referee and editors for a thorough reading and for suggestions which led to significant  improvements in the presentation of the paper. The work of both authors was partly funded by the ANR through project FLAIR (ANR-17-CE40-0012). 

\section{Group theoretical results}\label{sec:groups}

%One can explicitly construct families of abelian extensions $L/K$ satisfying the hypotheses of Theorem~\ref{theorem main}, as we now prove.
%We now prove Proposition~\ref{proposition main}.
The goal of this section is to construct families of abelian extensions $L/K$ satisfying the hypotheses of Theorem~\ref{theorem main}.

\begin{proof}[Proof of Proposition~\ref{proposition main}]
For $n\geq 8$, consider the permutations $g_1:=(12)(34)$ and $g_2:=(57)(68)$ as elements of $S_n$. Let $G:= \langle (12)(34),(5678) \rangle  < S_n$. We claim that the choices $C_1={g_1}$ and $C_2={g_2}$ satisfy the required properties. Indeed, 
$ C_1^+ = C_2^+=C_{(2,2)} $, where $C_{(2,2)}$ is the set of elements of $S_n$ of cycle type $(2,2)$. Moreover, an enumeration of the elements of $G$ shows that $r_G(g_1)=0$ and $r_G(g_2)=4$.
\end{proof}

The next lemma gives a 
%We give below a
 group theoretical criterion which generalizes the construction in the proof of Proposition~\ref{proposition main} and which implies the conditions of Theorem~\ref{theorem main}. (Here and later in the paper we make a slight abuse of notation by denoting $r_G(C)$ the common value $r_G(g)$ as $g$ runs over the $G$-conjugacy class $C$.)

\begin{lemma}
Let $G^+$ be a group and let $H$ and $K$ be subgroups having trivial intersection and such that $H$ centralizes $K$. Let $h\in H$ be a non-square (in $H$), and let $k\in K$ be a square (in $K$) which is a conjugate of $h$ in $G^+$. Then, the conjugacy classes $C_1=C_h$ and $C_2=C_k$ in the group $G=HK$ are such that $r_G(C_2)> r_G(C_1) $; in other words, the conditions of Theorem~\ref{theorem main} hold.

\label{lemma criterion} 
\end{lemma}

\begin{proof}
The fact that $H$ centralizes $K$ guarantees that $G=HK=KH$ is a subgroup of $G^+$. Moreover, any $x\in G$ such that $x^2=k$ can be written $x=st$ with $s\in H$ and $t\in K$ (and in this decomposition there is a unique $(s,t)$ corresponding to each $x$ since 
$H\cap K=\{1\}$). Thus $k=s^2t^2$, which implies that $s^2\in H\cap K$. Therefore $s^2=1$, and as a result
$$
\#\{x\in G\colon  x^2=k\}=\#\{x\in K\colon  x^2=k\}\cdot 
\#\{x\in H\colon  x^2=1\}>0\,.$$
By symmetry, we also have that 
$$\#\{x\in G\colon  x^2=h\}=\#\{x\in H\colon  x^2=h\}\cdot 
\#\{x\in K\colon  x^2=1\}=0\,.
$$ 
\end{proof}

In order to apply Lemma~\ref{lemma criterion}, 
%allows for a generalization of the example given in Section~\ref{section intro}. For instance, 
take for instance $G^+=S_n$, and let $\sigma,\tau \in S_n$ be permutations of order divisible by $4$ which have the same cycle type, but have disjoint supports. Consider the subgroups $H=\langle \sigma^2 \rangle$ and $K=\langle \tau \rangle$, and the elements $h=\sigma^2$ and $k=\tau^2$. We clearly have that $ r_K(k)\geq 1$ and $r_{H}(h)=0 $, and Lemma~\ref{lemma criterion} applies. 

% For the group $S_n$, the inverse Galois problem over the rationals has been known since Hilbert, and thus we reach the following conclusion.

\begin{remark*}
From a group theoretical point of view, it would be  interesting to classify the tuples $(G,G^{+},C_1,C_2)$ such that $G^+$ is a finite group, $G <G^+$ and $C_1, C_2$ are conjugacy classes of $G$ such that $r_G(C_1)\neq r_G(C_2)$ and $C_1^+=C_2^+$ in $G^+$ (recall~\eqref{equation definition C+}). For example, one notices that no such tuple exists where $G$ is a normal subgroup of $G^+$ (see~\cite[Proof of Lemma 3.13]{FJ}). Beyond this case, one may ask the following questions: how rare is the property enjoyed by these tuples?
% when one randomly picks $G,H,C_1,C_2$, 
What are the ``minimal'' examples?
% of such tuples?
%in terms of order, index,...), etc... 
Such questions are the subject of Mounir Hayani's forthcoming Ph.D. thesis.
%
%of these questions have in fact already been settled by Mounir Hayani and the results will appear as part of his forthcoming Ph.D. Thesis.
\end{remark*}

\section{Proof of Theorem~\ref{theorem main}}\label{sec:mainproof}

%The proof of Theorem~\ref{theorem main} borrows from our previous work~\cite{FJ}, but we will try to be relatively self-contained. In particular 

To introduce the natural framework of Theorem~\ref{theorem main}, we will work in the setting of~\cite{Be}, that is we will consider general class functions $t\colon\Gal(L/K)\rightarrow \mathbb C$, and define\footnote{See for instance~\cite[Chap. 1 \S4]{Mar} for a definition of ${\rm Frob}_\mathfrak p$ in the case where $\mathfrak p$ is ramified.}
$$ \psi(x;L/K,t):= \sum_{\substack{\mathfrak p\triangleleft \mathcal O_K\\ \mathcal N\mathfrak p \leq x \\ k\geq 1 }} t({\rm Frob}_\mathfrak p^k) \log (\mathcal N\mathfrak p);\qquad \theta(x;L/K,t):= \sum_{\substack{\mathfrak p\triangleleft \mathcal O_K\\ \mathcal N\mathfrak p \leq x }} t({\rm Frob}_\mathfrak p)\log (\mathcal N\mathfrak p); $$
$$ \pi(x;L/K,t):= \sum_{\substack{\mathfrak p\triangleleft \mathcal O_K\\ \mathcal N\mathfrak p \leq x  \\ \mathfrak p \text{ unram.}}} t({\rm Frob}_\mathfrak p). $$
When $L/\Q$ is Galois, we will use the shorthands $G:=\Gal(L/K)$, $G^+:=\Gal(L/\Q)$, as well as
$$ t^+ = {\rm Ind}_G^{G^+} t \colon G^+\to \mathbb C\,,\qquad g\mapsto   \sum_{\substack{aG\in G^+/G : \\a^{-1} g a \in G }} t(a^{-1} g a). 
 $$
 Finally, we recall that the inner product of class functions $t_1,t_2\colon G \rightarrow \mathbb C$ is defined by 
 $$ \langle t_1 , t_2 \rangle_G := \frac 1{|G|} \sum_{g\in G} t_1(g)\overline{t_2(g)}. $$
(We will simply write $\langle t_1,t_2\rangle$, dropping the subscript $G$, where the underlying group is clear from context.)

\begin{lemma}
\label{lemma mobius and sbp}
Let $L/K$ be an extension of number fields for which $L/\Q$ is Galois, and let $t\colon\Gal(L/K)\rightarrow \mathbb C$ be a class function. We have the estimate
$$ \pi(x;L/K,t) = \int_{2^-}^x  \frac{\d \psi(u;L/\Q,t^+) }{\log u}  -\langle t,r_G \rangle \frac{x^{\frac 12}}{\log x}+ o\Big(\frac{x^{\frac 12}}{\log x}\Big). $$
\end{lemma}

\begin{proof}
For any integer $\ell\geq 2$, denote by $f_\ell\colon G\to G$ the class function defined by $f_\ell(g)=g^\ell$. Let $\mu$ denote the M\"obius function; inclusion-exclusion implies that
\begin{align*}
 \theta(x;L/K,t) &=\psi(x;L/K,t) + \sum_{\ell \geq 2} \mu(\ell) \psi(x^{\frac 1\ell};L/K,t\circ f_\ell) \\
 &= \psi(x;L/K,t)- \langle t,r_G \rangle x^{\frac 12}(1+o(1))+O(x^{\frac 13}),  
\end{align*}
by the Chebotarev density theorem and the identity $ \frac 1{|G|}\sum_{g\in G} t(g^2) = \langle t,r_G\rangle . $  
%
% The Fourier decomposition $t=\sum_{\chi \in \Irr(G)}\langle t,  \chi\rangle$ implies that 
%$$  $$
The claimed estimate follows from a summation by parts and an application of the identity $$\psi(u;L/K,t)=\psi(u;L/\Q,t^+),$$  which is a consequence of the invariance of Artin $L$-functions under induction~(\cite[\S2]{Artin}, under the form used in~\cite[Proposition 3.11]{FJ}).
\end{proof}

\begin{proof}[Proof of Theorem~\ref{theorem main}]
We first compute, for any conjugacy class $C$ of $G$, any fixed $g_C\in C$ and any irreducible character $\chi$ of $G^+$,
$$
\langle 1_{C}^+,\chi\rangle_{G^+}=
\langle 1_C,\chi_{|G}\rangle_G=\frac{|C|}{|G|}\overline{\chi(g_C)}=\frac{|C| |G^+|}{|G| |C^+|}\langle 1_{C^+},\chi\rangle_{G^+}\,,
$$
where the first step uses Frobenius reciprocity. Therefore, denoting $t_{C_1,C_2}\colon\Gal(L/K)\rightarrow \mathbb C $ the class function $t_{C_1,C_2}=\frac{|G|}{|C_1|}1_{C_1}-\frac{|G|}{|C_2|}1_{C_2}$, one has $t_{C_1,C_2}^+ =\frac{|G^+|}{|C_1^+|}1_{C_1^+}-\frac{|G^+|}{|C_2^+|}1_{C_2^+}  \equiv 0$. Hence, Lemma~\ref{lemma mobius and sbp} implies that
$$ \pi(x;L/K,t_{C_1,C_2}) =-\langle t_{C_1,C_2},r_G \rangle \frac{x^{\frac 12}}{\log x} + o\Big(\frac{x^{\frac 12}}{\log x}\Big). $$
However, $-\langle t_{C_1,C_2},r_G \rangle = r_G(g_{C_2})-r_G(g_{C_1})>0 $, and thus $ \pi(x;L/K,t_{C_1,C_2}) >0 $ for all large enough values of $x$.
\end{proof}

%\section{Concluding remarks}

%In this final section w
We now discuss more precisely the oscillations of $\pi(x;L/K,C_1) -\pi(x;L/K,C_2)$ for triples $(L/K,C_1,C_2)$ chosen as in the proof of Proposition~\ref{proposition main} and Corollary~\ref{cor} (an explicit example of such a Galois extension produces Figure~\ref{figure}, and the purpose here is to discuss the rate of convergence of the function plotted to its asymptotic value). 
We recall that in the proof of Proposition~\ref{proposition main}, we have chosen $G=\langle  (12)(34),(5678) \rangle $ and $t=1_{C_1}-1_{C_2} $, where  $C_1=\{(12)(34)\}$ and $C_2=\{(57)(68)\}$. Since $G$ is abelian of order $8$, the class function  $ r_m(g):=\#\{ h\in G \colon h^m = g \}$ is identically equal to $1$ for all odd $m\geq 1$, 
%Moreover, the choices $g_1=(12)(34)$ and $g_2=(57)(68)$ are such that for $k\in \mathbb N $, $r_{4k}(g_1)=r_{4k}(g_2)=0$, $r_{4k+2}(g_1) =r_{2}(g_1)=0 $ and  $r_{4k+2}(g_2)=r_{2}(g_2)=4$. 
 and in particular, $\langle t\circ f_3,1\rangle = 0$ (where we recall that $f_\ell$ is the function on $G$ raising elements to their $\ell$-th power). The identity $\psi(x;L/K,t)=\psi(x;L/\Q,t^+)\equiv 0$ and the Riemann Hypothesis for Artin $L$-functions then imply that
 %Chebotarev density theorem then imply that 
\begin{align*}
 \theta(x;L/K,t) &=\psi(x;L/K,t) - \psi(x^{\frac 12};L/K,t\circ f_2)- \psi(x^{\frac 13};L/K,t\circ f_3) +O(x^{\frac 15})  \\ 
 &= -\langle t, r_G\rangle x^{\frac 12} + \sum_{\chi\in \Irr(G)} \overline{\langle \chi,t\circ f_2 \rangle} \sum_{\rho_\chi} \frac{x^{\frac 14+\frac{1}2 \Im(\rho_\chi) i}}{\rho_\chi}+O(x^{\frac 15}), 
 \end{align*}
by the explicit formula (see for instance~\cite[Theorem 3.4.9]{Ng}). Here, $\Irr(G)$ denotes the set of irreducible characters of $G$, and $\rho_\chi$ runs through the non-trivial zeros of the Artin $L$-function $L(s,L/K,\chi)$. Now, in this particular example $8 t\circ f_2 = 1_{\{(5678) \}}+1_{\{(5678)(12)(34) \}}+1_{\{(5876) \}}+1_{\{(5876)(12)(34) \}}$, and thus $\langle \chi, t\circ f_2 \rangle = \chi((5678))+\chi((5678)(12)(34))+ \chi((5876))+ \chi((5876)(12)(34)) $ (which is not identically zero). This explains why we expect the difference between the solid line and the data in Figure~\ref{figure} to be roughly of order $x^{-\frac 14}.$ (More precisely, we expect order $x^{-\frac 14}$ almost everywhere, and maximal order $x^{-\frac 14} (\log\log\log x)^2.$)

\section{Numerical examples}\label{sec:numerics}

\begin{figure}[hbt!]
%[h]

\caption{The normalized difference $(\pi(x;L/K,C_1) -\pi(x;L/K,C_2))/R(x) $ with $1\leq x \leq 10^{10}$ (data due to B. Allombert)}
%\begin{flushleft}
\centerline{\includegraphics[scale=.40]{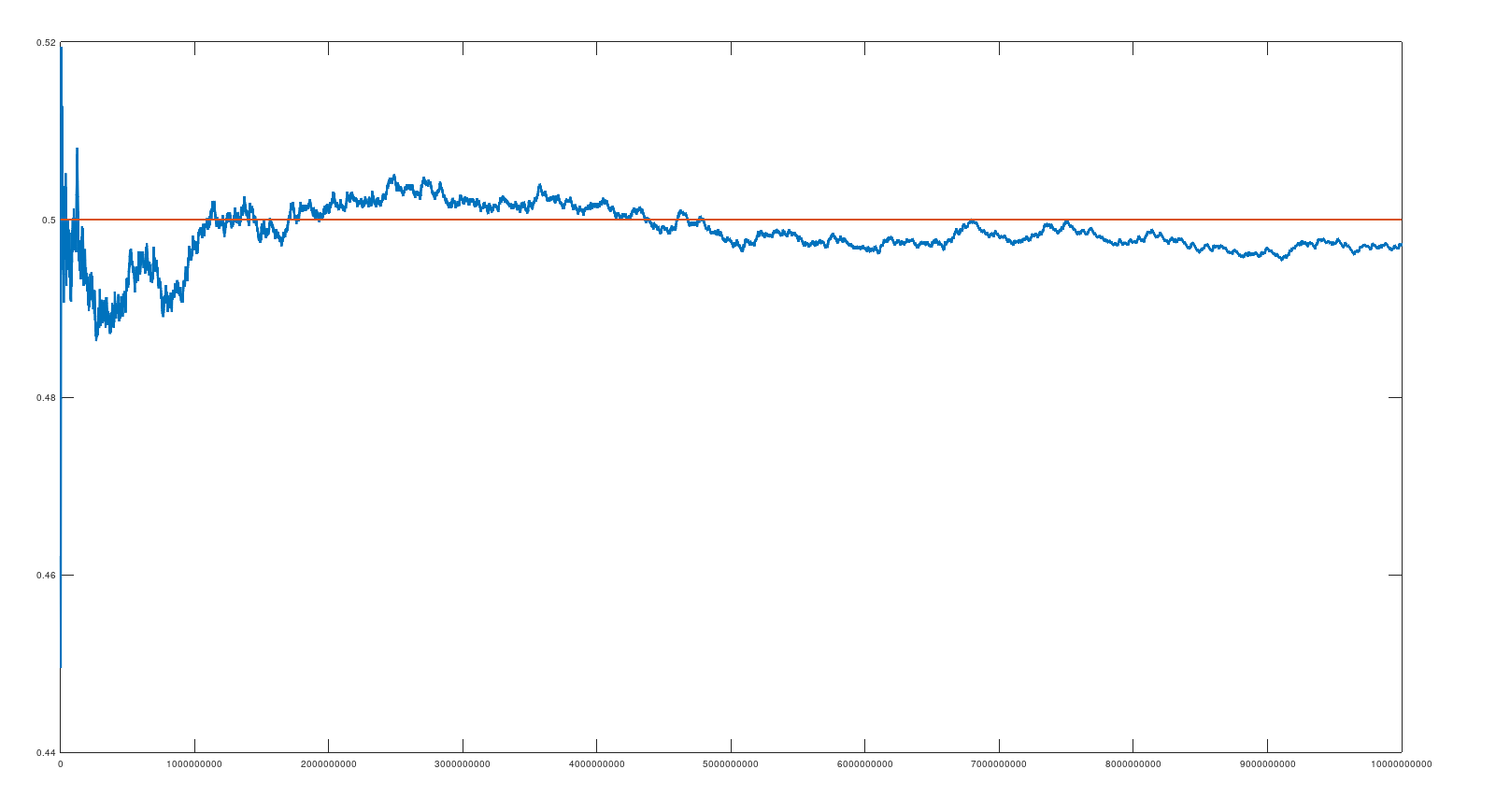}}
%\includegraphics[scale=.45]{ChebotarevBias-revised.png}
%\end{flushleft}
\label{figure}
\end{figure}

In this section we discuss our numerical verification of Theorem~\ref{theorem main} and Proposition~\ref{proposition main}. It would be computationally very expensive to work with the full group $S_8$. However, it turns out that one can replace $S_8$ with a relatively small subgroup which has the required properties.
% instead of  $S_8$ which has the required properties. 
Consider $G^+:=\langle (12)(34),(5678),(15)(27)(36)(48) \rangle $; let us show that $G^+$ is isomorphic to the wreath product of $\Z/4\Z$ and $\Z/2\Z$, which is of order $32$. 
%$G^+=\langle (12)(34),(5678),(15)(27)(36)(48))\rangle$ and 
Denote the permutations appearing in the generating set of $G^+$ by $\tau$, $\sigma$, and $\gamma$, respectively, and note that  
$G^+=\langle \sigma,\gamma\sigma\gamma,\gamma\rangle$  
%The inclusion $\supseteq$ is clear. Conversely, note that 
(since $\gamma\sigma\gamma=(1324)$, and thus $(\gamma\sigma\gamma)^2=\tau$). 
%This proves the inclusion $\subseteq$. 
The subgroup $\langle\sigma,\gamma\sigma\gamma\rangle$ is clearly isomorphic to $(\Z/4\Z)\times (\Z/4\Z)$. 
%, since the permutations $\sigma$ and $\gamma\sigma\gamma$ have order $4$ and disjoint supports. 
Moreover, conjugating  by $\gamma$ on $\langle\sigma,\gamma\sigma\gamma\rangle$ amounts to exchanging the two factors $\Z/4\Z$, which is the definition of the wreath product. 
%Therefore we have determined the group structure of $G^+$: this group is isomorphic to $\big((\Z/4\Z)\times (\Z/4\Z)\big)\rtimes (\Z/2\Z)$, where the action of the $\Z/2\Z$ factor permutes the two factors $\Z/4\Z$. In other words $G^+$ is the wreath product of $\Z/4\Z$ by $\Z/2\Z$ (its "GAP id" is (32,11) [REF]...).

Consider also the abelian subgroup $G:= \langle (12)(34),(5678) \rangle <G^+ $ as well as the conjugacy classes $C_1:=\{(12)(34)\}$ and $C_2:=\{(57)(68)\}$.
%defined in the last paragraph still
In the group $G^+$, one has that $\gamma^{-1} C_1 \gamma =C_2$, that is $C_1^+=C_2^+$. It follows from Theorem~\ref{theorem main} that for any Galois number field $L/\Q$ such that $\Gal(L/\Q) \simeq G^+$, the sub-extension $K=L^{G}$ has the property that for all large enough $x$,
$$ \pi(x;L/K,C_1) >\pi(x;L/K,C_2)$$(recall that $|C_1|=|C_2|=1$). Bill Allombert has kindly provided us with the {\tt pari/gp} code allowing for a  numerical check of this inequality up to $x=10^{10}$, for a particular number field $L/\Q$ of Galois group $G^+$. Explicitly, $L=\Q[x]/(f(x))$, where\footnote{ For the full code, \href{https://www.math.u-bordeaux.fr/~fjouve001/UnconditionalBiasCode.gp}{click here} or visit https://www.math.u-bordeaux.fr/$\sim$fjouve001/UnconditionalBiasCode.gp.} 
\begin{align*}
f(x)=&x^{32} - 128 x^{30} + 5680 x^{28} - 120576 x^{26} + 1386352 x^{24} - 9267712 x^{22} + 38233408  x^{20}\\
& - 101305344 x^{18} + 176213088 x^{16} - 202610688 x^{14} + 152933632 x^{12} - 74141696 x^{10} \\
&+ 22181632 x^8 - 3858432 x^6 + 363520 x^4 - 16384 x^2 + 256\,.
\end{align*}
In Figure~\ref{figure} we have plotted the difference $\pi(x;L/K,C_1) -\pi(x;L/K,C_2)$, normalized by the function 
$$R(x):= \frac {x^{\frac 12}}{\log x}+ \int_2^{x} \frac{\d u}{u^{\frac 12} (\log u)^2}\sim \frac {x^{\frac 12}}{\log x},$$
which can be shown following the proof of Lemma~\ref{lemma mobius and sbp} to be the ``natural approximation'' for the order of magnitude of this difference. As expected, we see that the plotted function converges to $\frac 12$, and to illustrate this we have added the solid line $y=\frac 12$ on the plot. Finally, we see that as predicted in Section~\ref{sec:mainproof}, the difference between the graph and the solid line is of order $x^{\frac 14}$.

\end{document}